\documentclass{article}

\RequirePackage[OT1]{fontenc}
\RequirePackage[%lnms,
amsthm,amsmath%,natbib
]{imsart}

\usepackage{graphicx}
\usepackage{latexsym,amsmath}
\usepackage{amsmath,amsthm,amscd}
\usepackage{amsfonts}
\usepackage[psamsfonts]{amssymb}
\usepackage{enumerate}
\usepackage{xcolor}

\usepackage{url}

\usepackage{tocvsec2}

\usepackage{epstopdf}
\usepackage{wrapfig}
\usepackage{float}

\startlocaldefs
\numberwithin{equation}{section}

\theoremstyle{plain} 
\newtheorem{theorem}{Theorem}[section]
\newtheorem{corollary}[theorem]{Corollary}

\newtheorem{proposition}[theorem]{Proposition}

\theoremstyle{definition} 

\theoremstyle{definition} 

\newtheorem*{ex*}{Example}
\theoremstyle{remark} 

\theoremstyle{remark} 

\newtheorem*{remark*}{Remark}
\numberwithin{equation}{section}

\reversemarginpar

\newcommand{\dd}{\partial}

\renewcommand{\dd}{{\,\operatorname{d}}}

\newcommand{\la}{\lambda}

\newcommand{\de}{\delta}
\newcommand{\be}{\beta}
\newcommand{\De}{\Delta}

\newcommand{\vpi}{\varphi}

\renewcommand{\Psi}{\overline{\Phi}}

\newcommand{\ii}[1]{\,\mathbf{I}\{#1\}}

\renewcommand{\P}{\operatorname{\mathsf{P}}} 
\newcommand{\E}{\operatorname{\mathsf{E}}}

\newcommand{\R}{\mathbb{R}}

\newcommand{\EE}{\mathcal{E}}

\newcommand{\bxi}[1]{{\xi_{w,#1}}}
\newcommand{\bS}{{S_w}}
\newcommand{\bK}{{\overline K}}

\renewcommand{\le}{\leqslant}
\renewcommand{\ge}{\geqslant}

\endlocaldefs

\begin{document}

\begin{frontmatter}

\title{Improved nonuniform Berry--Esseen-type bounds}
\runtitle{Berry--Esseen bound}
%\date{\today}

% \author{\fnms{First}  \snm{Author}\corref{}\thanksref{t2}\ead[label=e1]{first@somewhere.com}},
%  \author{\fnms{Second} \snm{Author}\ead[label=e2]{second@somewhere.com}}
%  \and
%  \author{\fnms{Third}  \snm{Author}%
%  \ead[label=e3]{third@somewhere.com}%
%  \ead[label=u1,url]{http://www.foo.com}}
%
%  \thankstext{t2}{Footnote to the first author with the `thankstext' command.}

\begin{aug}
\author{\fnms{Iosif} \snm{Pinelis}\thanksref{t2}\ead[label=e1]{ipinelis@mtu.edu}}
  \thankstext{t2}{Supported in part by NSF grant DMS-0805946
  }
\runauthor{Iosif Pinelis}

%\affiliation{Michigan Technological University}

\address{Department of Mathematical Sciences\\
Michigan Technological University\\
Houghton, Michigan 49931, USA\\
E-mail: \printead[ipinelis@mtu.edu]{e1}}
\end{aug}

\begin{abstract} 
New nonuniform Berry--Esseen-type bounds for sums of independent random variables are obtained, motivated by recent studies concerning such bounds for nonlinear statistics. 
The proofs are based on the Chen--Shao concentration techniques within the framework of the Stein method. 
%The results generalize and improve existing ones.  
\end{abstract}

%\subjclass[2000]{60E15, 62G10, 62G15, 60G50, 62G35}
% 62G10    	Hypothesis testing
%  62G15    	Tolerance and confidence regions
%  60G50    	Sums of independent random variables; random walks
%   62G35    	Robustness
  
%
%\keywords{probability inequalities; Rade\-macher random variables; sums of independent random variables; Student's test; self-normalized sums}


%60E15   	Inequalities; stochastic orderings
%62E17   	Approximations to distributions (nonasymptotic)
%62H10   	Distribution of statistics
%62H15   	Hypothesis testing [multivar]
%62F03   	Hypothesis testing [param]
%62G10   	Hypothesis testing [nonpar]


\end{frontmatter}

\settocdepth{chapter}

\tableofcontents 
%%%%%%%%%%%%%%%%%{\small\tableofcontents} 

\settocdepth{subsubsection}

\theoremstyle{plain} 
\numberwithin{equation}{section}

\section{Summary and discussion}\label{intro} 

For any natural $n$, let $\xi_1,\dots,\xi_n$ be independent zero-mean random variables (r.v.'s) such that 
\begin{equation}\label{eq:=1}
	\E\xi_1^2+\dots+\E\xi_n^2=1, 
\end{equation}
and let 
\begin{equation*}
	S:=\xi_1+\dots+\xi_n.  
\end{equation*} 

Take any $v\in(0,\infty)$ and introduce  
\begin{equation*}
	\be_v:=\E g(\xi_1/v)+\dots+\E g(\xi_n/v),
\end{equation*}
where 
\begin{equation}\label{eq:g} 
	g(x):=x^2\wedge|x|^3% \quad\text{and}\quad 2\le p\le3,  
\end{equation}
for all $x\in\R$. 
Observe that 
for any $p\in[2,3]$ one has $g(x)\le|x|^p$ for all real $x$ and hence 
\begin{equation*}
\be_v\le\mu_p/v^p, 	
\end{equation*}
where 
\begin{equation}\label{eq:mu_p}
	\mu_p:=\E|\xi_1|^p+\dots+\E|\xi_n|^p  
\end{equation}
for all $p\in(0,\infty)$. 
 
Take also any $w\in(0,\infty)$ and let $\bxi 1,\dots,\bxi n$ be any independent r.v.'s such that for each $i=1,\dots,n$ 
\begin{equation*}%\label{eq:bxi}
	\text{$\bxi i=\xi_i$ on the event $\{\xi_i\le w\}$, and $-\xi_i\le\bxi i\le w$ on the event $\{\xi_i>w\}$;} 
\end{equation*}
for instance, this condition will be satisfied if $\bxi i$ is defined as $\xi_i\wedge w$ (a Winsorization of $\xi_i$%, which is the largest possible version of $\bxi i$
) or as $\xi_i\ii{\xi_i\le w}$ (a truncation of $\xi_i$); cf.\ \cite{winzor}; note that 
\begin{equation}\label{eq:|bxi|<|xi|}
	\bxi i\le\xi_i\wedge w\quad\text{and}\quad|\bxi i|\le|\xi_i|. 
\end{equation}   
Introduce next  
\begin{equation*}
	\bS:=\bxi 1+\dots+\bxi n. 
\end{equation*}
Let $Z$ stand for a standard normal r.v. 
Let $A$, possibly with subscripts, denote positive real constants depending only on the values of the corresponding parameters; more precisely, let $A$ stand for an 
expression which 
takes positive real values 
depending -- in a continuous manner -- only on the values of the subscripts; let us allow such expressions to be different in different contexts.  
For any two expressions $\EE_1$ and $\EE_2$, let us also write $\EE_1\le_\bullet\EE_2$ in place of $\EE_1\le A_\bullet\,\EE_2$, where $\bullet$ stands for the relevant subscript(s). 

Take also any $\la\in(0,\infty)$. 

\begin{theorem}\label{th:}
%Take any $\la\in(0,\infty)$. 
For all $z\in\R$ 
\begin{equation}\label{eq:BE}
|\P(\bS>z)-\P(Z>z)|
  \le_{v,w,\la}\frac{\be_v}{e^{\la z}}. 
\end{equation} 
\end{theorem} 

The necessary proofs will be given in Section~\ref{proofs}. 

Taking \eqref{eq:BE} with $\bxi i=\xi_i\ii{\xi_i\le w}$, $v=w=1$ and $\la=1/2$, and then replacing there $\be_1$ by its upper bound $\mu_3$, one has, as a corollary, the result obtained by Chen and Shao \cite[(6.15)]{chen-shao05}. 

An advantage of such bounds is that they decrease fast as $z\to\infty$. 
Moreover, if the right tails of (the distributions of) the r.v.'s $\xi_i$ are light enough, then $\P(\bS>z)$ will differ little from $\P(S>z)$ for large $z$. 
This observation can be formalized in a variety of ways, some of which are presented in the following proposition.

\begin{proposition}\label{prop:osipov} 
For all positive $p$, $w$, $c$, $y$, and $z$ %$p\in(0,\infty)$ and $z\in[0,\infty)$ 
\begin{align}\label{eq:diff}
0\le\P(S>z)-\P(\bS>z)
  \le P_1\wedge\dots\wedge P_5, %\quad\text{where} \\  
%  _{p,v,w,\la}\frac{\be_v}{e^{\la z}}+\frac{\mu_p}{(1+z)^p}. 
\end{align} 
where 
\begin{align}
P_1&:=\P(\max_i\xi_i>w), \notag \\ 
P_2&:=\P(\max_i\xi_i>y)+Q(z,y)\sum_i\P(\xi_i>w), \notag \\ 
Q(z,y)&:=\max_i\P(S-\xi_i>z-y,\;\max_{j\ne i}\xi_j\le y), \notag \\ 
P_3&:=\P(\max_i\xi_i>y)+2Q_*(z,y)\P(\max_i\xi_i>w), \label{eq:P3} \\ 
Q_*(z,y)&:=Q(z,y)\vee\P(S>z,\;\max_j\xi_j\le y), \notag \\ 
P_4&:=\P\Big(\max_i\xi_i>\frac z{1+p/2}\Big)+\frac{A_{p,c}}{(c+z)^p}\P(\max_i\xi_i>w), \notag \\ 
P_5&:=\frac{A_{p,w,c}\mu_p}{(c+z)^p}.  \notag 
\end{align} 
\end{proposition} 

Upper bounds on $Q(z,y)$ and $Q_*(z,y)$ can be obtained using exponential bounds on large deviation probabilities such as ones due to Bennett \cite{bennett}, Hoeffding \cite{hoeff63}, and Pinelis and Utev \cite{pin-utev-exp}; in fact, the Bennett--Hoeffding inequality was used to obtain the bounds $P_4$ and $P_5$ defined above. One can also use bounds that are better than the corresponding exponential ones, such as ones in \cite{pin98,bent-ap,pin-hoeff}. 

Combining \eqref{eq:BE} and \eqref{eq:diff}, one immediately obtains 

\begin{corollary}\label{cor:pipiras}
For all positive $p$, $v$, $w$, $\la$, $c$, and $z$ %$p\in(0,\infty)$ and $z\in[0,\infty)$ 
\begin{align}\label{eq:pip}
|\P(S>z)-\P(Z>z)|
  \le_{p,v,w,\la,c}\frac{\be_v}{e^{\la z}}+(P_1\wedge\dots\wedge P_5). 
\end{align} 
\end{corollary} 

An inequality similar to $|\P(S>z)-\P(Z>z)|\le_{p,v,w,\la,c}\frac{\be_v}{e^{\la z}}+P_5$, but with $\frac{\mu_3}{(1+z)^p}$ in place of $\frac{\be_v}{e^{\la z}}$, was obtained for $p\ge3$ by Pipiras \cite{pipiras}, which in turn is an extension of the corresponding result by Osipov \cite{osip67} for identically distributed $\xi_i$'s. 

In the ``i.i.d.'' case, when the $\xi_i$'s are copies in distribution of a r.v.\ $X/\sqrt n$, with $\E X=0$ and $\E X^2=1$, the term $\frac{\be_v}{e^{\la z}}$ in \eqref{eq:pip} is $\le_v\frac{\E|X|^3}{\sqrt n}\,e^{-\la z}$, which decreases fast in $z$; at that, up to a constant factor, the bound $P_5$ is $\frac{\E|X|^p}{(c+z)^pn^{p/2-1}}$, which decreases fast both in $z$ and $n$ if $p$ is taken to be large and $\E|X|^p<\infty$.  

Of the other known Berry--Esseen-type bounds, the one closest to \eqref{eq:BE} and \eqref{eq:pip} in form is apparently as follows: 
\begin{equation}\label{eq:bikelis}
|\P(S>z)-\P(Z>z)|
  \le A\sum_{i=1}^n\E\Big(\frac{\xi_i^2}{(|z|+1)^2}\wedge\frac{|\xi_i|^3}{(|z|+1)^3}\Big)
  \le_v\be_v. 
\end{equation} 
This result was obtained in a slightly more general form by Bikelis \cite[Theorem~4]{bik66} (see also \cite[Chapter V, Supplement 24]{pet75}), and in its present form by Chen and Shao \cite[Theorem 2.2]{chen01}. 

The improvement of the existing results provided by Theorem~\ref{th:} of this note was needed in \cite{nonlinear} to determine the optimal zone of large deviations in which a certain nonuniform Berry--Esseen-type bound for nonlinear statistics holds. 

\section{Proofs}\label{proofs} 

An important role in the proof of Theorem~\ref{th:} is played by the following upper bound on the concentration probability. 

\begin{proposition}\label{prop:concentr}
For all real numbers $a$ and $b$ such that $a<b$ and all $i=1,\dots,n$
\begin{equation}\label{eq:prop6.1}
	\P(a\le S_w-\bxi i\le b)\le_{v,w,\la}(b-a+\be_v)e^{-\la a}.   
\end{equation} 
\end{proposition} 

In the case when $\bxi i=\xi_i\ii{\xi_i\le w}$, $v=w=1$, and $\la=1/2$, 
an inequality similar to \eqref{eq:prop6.1}, with $\mu_3$ in place of $\be_v$, was obtained by Chen and Shao \cite[Proposition~6.1]{chen-shao05}. 

%Taking \eqref{eq:BE} and \eqref{eq:prop6.1} with $v=w=1$ and $\la=1/2$, and then replacing there $\be_1$ by its upper bound $\mu_3$, one has the results obtained by Chen and Shao \cite[(6.15) and Proposition~6.1]{chen-shao05}. 

\begin{proof}[Proof of Proposition~\ref{prop:concentr}] 
This proof mainly follows the lines of the proof of the mentioned Proposition~6.1 in \cite{chen-shao05}. 
Just as there, one can use the Bennett--Hoeffding inequality \cite[(6.2)]{chen-shao05} to see that without loss of generality (w.l.o.g.) $\be_v$ is small enough, say $\be_v\le0.5/v^2$; cf.\ the main case $\mu_3\le0.17$ in the proof of Proposition~6.1 in \cite{chen-shao05}.   
Introduce now 
\begin{equation*}
	\de:=v^3\be_v/2; 
\end{equation*}
cf.\ \cite[Remark~2.2]{chen07}. 
Then the condition $\be_v\le0.5/v^2$ yields $\de\in(0,v/4]$, and so, 
the inequality $x(x\wedge d)\ge x^2-\frac{v^3}{4d}\,g(\frac xv)$ for all $x\in[0,\infty)$ and $d\in(0,v/4]$ implies that  
\begin{equation}\label{eq:>1/2}
\sum_1^n\E|\xi_j|(|\xi_j|\wedge\de)\ge1-\frac{v^3}{4\de}\,\be_v=\frac12;  	
\end{equation}
here we also used \eqref{eq:=1}. 
Therefore (cf.\ \cite[(6.8)]{chen-shao05}) and in view of the inequality $|x|\ii{x>w}\le kg(\frac xv)$ for $k:=\frac{v^2}{w^2}\,(v\vee w)$ and all $x\in\R$, 
\begin{align*}
	\sum_{j\ne i}\E|\xi_j|(|\bxi j|\wedge\de)
	&\ge\sum_{j\ne i}\E|\xi_j|[(|\xi_j|\wedge\de)-\de\ii{\xi_j>w}] \\ 
	&\ge-\de\E|\xi_i|-k\de\be_v+\sum_1^n\E|\xi_j|[(|\xi_j|\wedge\de) \\
	&\ge-\de-k\de\be_v+0.5
	=-\tfrac12\,v^3\be_v(1+k\be_v)+0.5 
	\ge0.4, 
\end{align*} 
if $\be_v$ is assumed (w.l.o.g.) to be small enough. 

The penultimate inequality in the last display follows by \eqref{eq:>1/2} and \eqref{eq:=1}, which latter implies $\E|\xi_i|\le\sqrt{\E\xi_i^2}\le1$. This trivial upper bound on $\E|\xi_i|$ can be improved, in an optimal way, as follows. Note that 
\begin{equation}\label{eq:young}
	u\le\frac{2k}3+\frac{g(u)}{3k^2} \quad\text{for all
	$u\in[0,\infty)$ and $k\in(0,\tfrac89]$,} 
\end{equation}
where the function $g$ is as defined in \eqref{eq:g}. 
This Young-type inequality can be quickly verified using the Mathematica command 
\verb9Reduce[ForAll[u, u >= 0,9 \verb9\[CapitalDelta][k,u] >= 0] && k > 0]9, which outputs \verb10 < k <= 8/91; here the argument of the command 
%\verb9ForAll[u, u >= 0, \[CapitalDelta][u] >= 0]9 
stands for the condition %\break 
$\big(\forall_{u\in[0,\infty)}\De(k,u)\ge0\big)\;\&\;$ $k\in(0,\infty)$, where 
$\De(k,u):=\tfrac{2k}3+\tfrac{g(u)}{3k^2}-u$. 
Alternatively, \eqref{eq:young} can be checked as follows. First, note that for all $u\in[0,\infty)$ and $k\in(0,\frac 89]$ one has $\De(k,u)\ge\De_*(u):=\De(k_u,u)$, where $k_u:=g(u)^{1/3}\wedge\frac89
=g\big(u\wedge \frac 89\big)^{1/3}$. 
If now $u\in[0,\frac 89]$ then $k_u=u$ and $\De_*(u)=0$; 
if $u\in[\frac 89,1]$ then $k_u=\frac89$ and $\De_*(u)=\break \frac{27}{64}(u-\frac89)^2(u+\frac{16}9)\ge0$; 
and if $u\in[1,\infty)$ then $k_u=\frac89$ and $\De_*(u)=\frac{27}{64}(u-\frac{32}{27})^2\ge0$. 
This proves \eqref{eq:young}. 
Assuming now, w.l.o.g., that $\be_v\le(\frac89)^3$ one can use \eqref{eq:young} with $k=\be_v^{1/3}\le\frac89$ to see that  $\E\frac{|\xi_i|}{v}\le\tfrac{2\be_v^{1/3}}3+\tfrac{\be_v}{3\be_v^{2/3}}=\be_v^{1/3}$, whence  
\begin{equation*}%\label{eq:<be}
\E|\xi_i|\le v\be_v^{1/3}. 	
\end{equation*}
The latter upper bound, $v\be_v^{1/3}$, is no greater than 
the bound $\mu_3^{1/3}$ on $\E|\xi_i|$ used in \cite{chen-shao05}. 
Moreover, the bound $v\be_v^{1/3}$ %on $\E|\xi_i|$ 
is the best possible, for any given $v\in(0,\infty)$. Indeed, let $n\ge2$, $|\xi_1|\equiv x$, and $|\xi_i|\equiv y$ for $i=2,\dots,n$, where $n\to\infty$, $x:=1/\big(1+(n-1)^{1/6}\big)^{1/2}$, and  $y:=x/(n-1)^{5/12}$; then \eqref{eq:=1} holds, $\be_v\le(\frac89)^3$ eventually, and $\E|\xi_i|\sim v\be_v^{1/3}$. 
%On considering the example when $n=1$, $|\xi_1|\equiv1$, and $v=1$, it is also seen that the bound $v\be_v^{1/3}$ on $\E|\xi_i|$ cannot in general be replaced by $\ka v\be_v^{1/3}$ with any constant factor $\ka<1$. 

The other modifications that one needs to make in the proof of \cite[Proposition~6.1]{chen-shao05} are rather straightforward. 
Thus, the proof of Proposition~\ref{prop:concentr} is complete. 
\end{proof}

\begin{proof}[Proof of Theorem~\ref{th:}] 
This proof mainly follows the lines of the proof of the mentioned inequality \cite[(6.15)]{chen-shao05}. First, in view of \eqref{eq:bikelis}, w.l.o.g.\ $z\ge0$. 
Also, just to simplify the presentation, assume, as in \cite{chen-shao05}, that $v=w=1$; the modifications needed for general $v$ and $w$ in $(0,\infty)$ are straightforward. 
One can write, as in \cite[(6.16)]{chen-shao05},   
\begin{equation}\label{eq:sum}
	\P(S_1\le z)-\P(Z\le z)=R_1+R_{2,1}+R_{2,2}+R_3,
\end{equation}
where 
\begin{align*}
	R_1&:=\sum_i\E\xi_i^2\ii{\xi_i>1}\E f'_z(S_1), \\
	R_{2,1}&:=\sum_i\int_{-\infty}^1[\P(S_1\le z)-\P(S_1-\xi_{1,i}+t\le z)]\bK_i(t)\dd t, \\
	R_{2,2}&:=\sum_i\int_{-\infty}^1[\E S_1f_z(S_1)-\E(S_1-\xi_{1,i}+t)f_z(S_1-\xi_{1,i}+t)]\bK_i(t)\dd t, \\ 
		R_3&:=\sum_i\E\xi_i\ii{\xi_i>1}\E f_z(S_1-\xi_{1,i}), 
\end{align*}
where $\bK_i(t):=\E|\xi_{1,i}|\,\ii{|t|\le|\xi_{1,i}|,\,t\xi_{1,i}>0}$,  
$f_z(s):=\Phi(z)r(s)\ii{s>z}+\Phi(-z)r(-s)\ii{s\le z}$ is the Stein function, $r(s):=\Phi(-s)/\vpi(-s)$ is the Mills' ratio,
and $\Phi$ and $\vpi$ are, respectively, the distribution and density functions of the standard normal distribution. 

The terms $R_1$ and $R_2$ are bounded as in \cite{chen-shao05}, but using the inequalities $\sum_i\E\xi_i\ii{\xi_i>1}\le\sum_i\E\xi_i^2\ii{\xi_i>1}\le\be_1$ and \cite[(6.2)]{chen-shao05} with $t=\la$, to get 
\begin{equation}\label{eq:R1+R3}
	|R_1|+|R_3|\le_\la\be_1 e^{-\la z}. 
\end{equation}
Next, use Proposition~\ref{prop:concentr} and the fact that $(z-t)\wedge(z-\xi_{1,i})\ge z-1$ for $t\le1$, 
to write 
\begin{equation}\label{eq:R21}
	|R_{2,1}|\le_\la e^{-\la z}\sum_i\E f_{2,1}(\xi_{1,i},\eta_{1,i}), 
\end{equation}
where $\eta_{1,i}$ is an independent copy of $\xi_{1,i}$ and 
\begin{equation}\label{eq:f21}
	f_{2,1}(x,y):=\int_{-\infty}^1\big(1\wedge(|x|+|t|+\be_1)\big)|y|\ii{|t|\le|y|}\dd t. 
\end{equation}
In turn, one can bound $f_{2,1}(x,y)$ in two ways: for all real $x$ and $y$
\begin{align*}
	f_{2,1}(x,y)&\le \int_{-\infty}^1|y|\ii{|t|\le|y|}\dd t\le2y^2\le2\be_1y^2+3(x^2\vee y^2)\quad\text{and}\\ 
	f_{2,1}(x,y)&\le \int_{-\infty}^1(|x|+|t|+\be_1)|y|\ii{|t|\le|y|}\dd t\le2\be_1y^2+3(|x|^3\vee|y|^3), 
\end{align*} 
so that 
$f_{2,1}(x,y)\le2\be_1y^2+3g(x)+3g(y)$, where $g$ is as in \eqref{eq:g}. 
Now \eqref{eq:R21}, \eqref{eq:=1}, and \eqref{eq:|bxi|<|xi|} yield 
\begin{equation}\label{eq:R21<}
	|R_{2,1}|\le_\la e^{-\la z}\be_1. 
\end{equation}

The term $R_{2,2}$ is bounded similarly to $R_{2,1}$. Here, instead of Proposition~\ref{prop:concentr}, one can use \cite[Lemma~6.5]{chen-shao05} (with the factor $A_\la e^{-\la x}$ in place of $Ce^{-z/2}$) and \cite[(2.7)]{chen-shao05}. 
%; note also that for each real $z$ the value of $sf_z(s)$ is increasing from $0$ to $\Phi(z)<1$ as $s$ increases from $-\infty$ to $\infty$.   
So, one obtains inequality \eqref{eq:R21} with $R_{2,1}$ and $f_{2,1}$ replaced by $R_{2,2}$ and $f_{2,2}$, respectively, where $f_{2,2}$ is obtained from $f_{2,1}$ by removing the summand $\be_1$ from \eqref{eq:f21}. 
Thus, one has \eqref{eq:R21<} with $R_{2,2}$ in place of $R_{2,1}$. 
On recalling also \eqref{eq:sum} and \eqref{eq:R1+R3}, %and \eqref{eq:R21<}, 
this completes the proof of Theorem~\ref{th:}. 
\end{proof}

\begin{proof}[Proof of Proposition~\ref{prop:osipov}]
%Assume w.l.o.g.\ that $z\ge0$, and suppose first that $(z+1)^pG_\eta(\frac{z+1}{\frac p2+1})\ge1$. 
The first inequality in \eqref{eq:diff} is obvious. 
Let, for brevity, 
\begin{equation*}
	\De_w(z):=\P(S>z)-\P(\bS>z). 
\end{equation*} 
The upper bound $P_1$ on $\De_w(z)$ is obvious. 

Next, write (cf.\ e.g.\ \cite[(6.13) and on]{chen-shao05} or \cite{a.nagaev69-I,a.nagaev69-II,pin81}) 
\begin{equation*}%\label{eq:ineqs}
\begin{aligned}
	\P(S>z)-\P(\bS>z)&\le\P(S>z,\max_i\xi_i>w) \\
	\le\P(\max_i\xi_i>y)&+\P(S>z,\;y\ge\max_i\xi_i>w) \\
	\le\P(\max_i\xi_i>y)&+\sum_i\P(S-\xi_i>z-y,\;\max_{j\ne i}\xi_j\le y) \P(\xi_i>w),   
%	\le\P(\max_i\xi_i>y)&+\Big(\frac e{1+zy}\Big)^{z/y}\sum_i \P(\xi_i>w) \\ 
\end{aligned}	
\end{equation*}
which implies that $\De_w(z)\le P_2$. 

It is also easy to see that 
\begin{equation*}
	\sum_i\P(\xi_i>w)\le\frac{\P(\max_i\xi_i>w)}{1-\P(\max_i\xi_i>w)}  
\end{equation*}
if $\P(\max_i\xi_i>w)\ne1$. 
So, in the case when $\P(\max_i\xi_i>w)\le1/2$ one has 
$\sum_i\P(\xi_i>w)\le2\P(\max_i\xi_i>w)$ and hence $\De_w(z)\le P_2\le P_3$. 
In the other case, when $\P(\max_i\xi_i>w)>1/2$, use the obvious inequalities 
\begin{align*}
	\De_w(z)\le\P(S>z)&\le\P(\max_i\xi_i>y)+\P(S>z,\;\max_j\xi_j\le y) \\ 
	&\le\P(\max_i\xi_i>y)+Q_*(z,y) \le P_3. 
%	 \\ 
%	&\le\P(\max_i\xi_i>y)+2Q_*(z,y)\P(\max_i\xi_i>w).  
\end{align*} 
Thus, in either case $\De_w(z)\le P_3$. 

Concerning the bound $P_4$, choose   
\begin{equation*}%\label{eq:y}
	y=\frac z{1+p/2}  
\end{equation*}
in \eqref{eq:P3} and use 
the Bennett--Hoeffding inequality (see e.g.\ \cite[Theorem~8.2]{pin94}) to write   
\begin{equation*}
	Q_*(z,y)
	\le\Big(\frac e{(z-y)y}\Big)^{(z-y)/y}\le_p\frac1{z^p}\le_p\frac1{(c+z)^p} 
\end{equation*}
in the case when $z\ge c$. So, in this case one has $\De_w(z)\le P_3\le P_4$. 
If now $z\in[0,c]$ then $1\le_{p,c}\frac1{(c+z)^p}$ and hence $\De_w(z)\le P_1\le P_4$. 
Thus, in either case $\De_w(z)\le P_4$. 

Finally, in the case when $z\ge c$, the inequality $\De_w(z)\le P_5$ follows by Chebyshev's inequality from $\De_w(z)\le P_4$ and \eqref{eq:mu_p}. 
If now $z\in[0,c]$ then $\De_w(z)\le P_1\le\frac{\mu_p}{w^p}
\le_{p,w,c}\frac{\mu_p}{(c+z)^p}$, so that in either case $\De_w(z)\le P_5$. 
\end{proof}

\bibliographystyle{abbrv}
%\bibliographystyle{ims}
%\bibliography{are.citations}
%\bibliography{citat}

%\bibliography{citations}

%\bibliography{C:/Users/iosif-home-2011/Dropbox/mtu/bib_files/citations}
\bibliography{C:/Users/iosif/Dropbox/mtu/bib_files/citations}

\def\cprime{$'$} \def\polhk#1{\setbox0=\hbox{#1}{\ooalign{\hidewidth
  \lower1.5ex\hbox{`}\hidewidth\crcr\unhbox0}}}
  \def\polhk#1{\setbox0=\hbox{#1}{\ooalign{\hidewidth
  \lower1.5ex\hbox{`}\hidewidth\crcr\unhbox0}}}
  \def\polhk#1{\setbox0=\hbox{#1}{\ooalign{\hidewidth
  \lower1.5ex\hbox{`}\hidewidth\crcr\unhbox0}}} \def\cprime{$'$}
  \def\polhk#1{\setbox0=\hbox{#1}{\ooalign{\hidewidth
  \lower1.5ex\hbox{`}\hidewidth\crcr\unhbox0}}}
  \def\polhk#1{\setbox0=\hbox{#1}{\ooalign{\hidewidth
  \lower1.5ex\hbox{`}\hidewidth\crcr\unhbox0}}}
\begin{thebibliography}{10}

\bibitem{bennett}
G.~Bennett.
\newblock Probability inequalities for the sum of independent random variables.
\newblock {\em J. Amer. Statist. Assoc.}, 57(297):33--45, 1962.

\bibitem{bent-ap}
V.~Bentkus.
\newblock On {H}oeffding's inequalities.
\newblock {\em Ann. Probab.}, 32(2):1650--1673, 2004.

\bibitem{bik66}
A.~Bikelis.
\newblock Estimates of the remainder term in the central limit theorem.
\newblock {\em Litovsk. Mat. Sb.}, 6:323--346, 1966.

\bibitem{chen01}
L.~H.~Y. Chen and Q.-M. Shao.
\newblock A non-uniform {B}erry-{E}sseen bound via {S}tein's method.
\newblock {\em Probab. Theory Related Fields}, 120(2):236--254, 2001.

\bibitem{chen-shao05}
L.~H.~Y. Chen and Q.-M. Shao.
\newblock Stein's method for normal approximation.
\newblock In {\em An introduction to Stein's method}, volume~4 of {\em Lect.
  Notes Ser. Inst. Math. Sci. Natl. Univ. Singap.}, pages 1--59. Singapore
  Univ. Press, Singapore, 2005.

\bibitem{chen07}
L.~H.~Y. Chen and Q.-M. Shao.
\newblock Normal approximation for nonlinear statistics using a concentration
  inequality approach.
\newblock {\em Bernoulli}, 13(2):581--599, 2007.

\bibitem{hoeff63}
W.~Hoeffding.
\newblock Probability inequalities for sums of bounded random variables.
\newblock {\em J. Amer. Statist. Assoc.}, 58:13--30, 1963.

\bibitem{a.nagaev69-I}
A.~V. Nagaev.
\newblock Integral limit theorems with regard to large deviations when
  {C}ram\'er's condition is not satisfied. {I}.
\newblock {\em Theor. Probability Appl.}, 14:51--64, 1969.

\bibitem{a.nagaev69-II}
A.~V. Nagaev.
\newblock Integral limit theorems with regard to large deviations when
  {C}ram\'er's condition is not satisfied. {II}.
\newblock {\em Theor. Probability Appl.}, 14:193--208, 1969.

\bibitem{osip67}
L.~V. Osipov.
\newblock Asymptotic expansions in the central limit theorem.
\newblock {\em Vestnik Leningrad. Univ.}, 22(19):45--62, 1967.

\bibitem{pet75}
V.~V. Petrov.
\newblock {\em Sums of independent random variables}.
\newblock Springer-Verlag, New York, 1975.
\newblock Translated from the Russian by A. A. Brown, Ergebnisse der Mathematik
  und ihrer Grenzgebiete, Band 82.

\bibitem{pin-hoeff}
I.~Pinelis.
\newblock On the {B}ennett-{H}oeffding inequality (preprint),
  ar{X}iv:0902.4058v1 [math.{PR}].

\bibitem{pin94}
I.~Pinelis.
\newblock Optimum bounds for the distributions of martingales in {B}anach
  spaces.
\newblock {\em Ann. Probab.}, 22(4):1679--1706, 1994.

\bibitem{pin98}
I.~Pinelis.
\newblock Optimal tail comparison based on comparison of moments.
\newblock In {\em High dimensional probability ({O}berwolfach, 1996)},
  volume~43 of {\em Progr. Probab.}, pages 297--314. Birkh\"auser, Basel, 1998.

\bibitem{winzor}
I.~Pinelis.
\newblock Exact lower bounds on the exponential moments of {W}insorized and
  truncated random variables.
\newblock {\em J. App. Probab.}, 48:547--560, 2011.

\bibitem{nonlinear}
I.~Pinelis and R.~Molzon.
\newblock Berry-{E}ss{\'e}{e}n bounds for general nonlinear statistics, with
  applications to {P}earson's and non-central {S}tudent's and {H}otelling's
  (preprint), ar{X}iv:0906.0177v1 [math.{ST}].

\bibitem{pin81}
I.~F. Pinelis.
\newblock A problem on large deviations in a space of trajectories.
\newblock {\em Theory Probab. Appl.}, 26(1):69--84, 1981.

\bibitem{pin-utev-exp}
I.~F. Pinelis and S.~A. Utev.
\newblock Sharp exponential estimates for sums of independent random variables.
\newblock {\em Theory Probab. Appl.}, 34(2):340--346, 1989.

\bibitem{pipiras}
V.~Pipiras.
\newblock The remainder terms of an asymptotic expansion of the distribution
  function of a sum of independent random variables.
\newblock {\em Litovsk. Mat. Sb.}, 10:135--159, 1970.

\end{thebibliography}

%\bibliography{C:/Users/Iosif/Documents/mtu_home01-30-10/bib_files/citations}
%\bibliography{C:/Users/Iosif/Documents/mtu_home12-22-08/bib_files/citations}

\end{document}